\documentclass[11pt]{article}

\usepackage{fullpage}
\usepackage{amsthm}
\usepackage{latexsym}
\usepackage{amssymb,amsmath}
\usepackage{color}
\usepackage{enumerate}
\usepackage{tikz}

\usetikzlibrary{arrows}

\newtheorem{theorem}{Theorem}[section]

\newtheorem{proposition}[theorem]{Proposition}

\newtheorem{corollary}[theorem]{Corollary}

\newtheorem{remark}[theorem]{Remark}

\numberwithin{equation}{section}

\DeclareMathOperator{\NLRM}{NLRM}

\DeclareMathOperator{\ls}{ls}

\DeclareMathOperator{\wt}{wt}

\DeclareMathOperator{\GF}{GF}

\newcommand{\Ppp}{{\mathbb P}}

\newcommand{\qchoose}[3]{\genfrac{[}{]}{0pt}{}{#1}{#2}_{#3}}

\newcommand{\vanish}[1]{}

\newcommand{\rg}[2]{RG(#1,#2)}

\newcommand{\join}{\vee}

\newcommand{\plusdots}{+ \cdots +}
\newcommand{\timesdots}{\times \cdots \times}

\newcommand{\bigcupdotalt}{\charfusion[\mathop]{\bigcup}{\cdot}}

\makeatletter
\def\moverlay{\mathpalette\mov@rlay}
\def\mov@rlay#1#2{\leavevmode\vtop{%
   \baselineskip\z@skip \lineskiplimit-\maxdimen
   \ialign{\hfil$\m@th#1##$\hfil\cr#2\crcr}}}
\newcommand{\charfusion}[3][\mathord]{
    #1{\ifx#1\mathop\vphantom{#2}\fi
        \mathpalette\mov@rlay{#2\cr#3}
      }
    \ifx#1\mathop\expandafter\displaylimits\fi}
\makeatother

\begin{document}

\title{$q$-Stirling identities revisited}

\author{{\sc Yue CAI},
          \hspace*{2 mm} {\sc Richard EHRENBORG}\hspace*{2 mm} and \hspace*{2 mm}
        {\sc Margaret A.\ READDY}}

\date{}
\maketitle

\begin{abstract}
We give combinatorial proofs of $q$-Stirling identities
using restricted growth words. 
This includes a poset theoretic proof of Carlitz's identity,
a new proof of the $q$-Frobenius identity of Garsia and Remmel
and of Ehrenborg's Hankel $q$-Stirling determinantal identity.
We also develop a two parameter generalization to unify 
identities of Mercier and include a symmetric function
version.
\end{abstract}

\section{Introduction}
\label{section_introduction}

           The classical Stirling number of the second kind $S(n,k)$
is the number of set partitions of $n$ elements into $k$ blocks.  The
Stirling numbers of the second kind first appeared in work of Stirling
in 1730, where he gave the Newton expansion of the functions $f(z) =
z^n$ in terms of the falling factorial basis~\cite[Page~8]{Stirling}.
Kaplansky and Riordan~\cite{Kaplansky_Riordan} found the combinatorial
interpretation that the Stirling number $S(n,k)$ enumerates the number
of ways to place $n-k$ non-attacking rooks on a triangular board of side
length $n-1$.  From later work of Garsia and Remmel, this is
equivalent to the number of set partitions of $n$ elements into $k$
blocks~\cite{Garsia_Remmel}.

The $q$-Stirling numbers of the second kind
arose from Carlitz's 
development of a $q$-analogue of
the Bernoulli numbers
and is predated by a problem of his involving
abelian groups~\cite{Carlitz,Carlitz_identity}.
There is a long history of studying set
partitions~\cite{Comtet,Garsia_Remmel,Leroux,Rota},
Stirling numbers of the second
kind and their $q$-analogues~\cite{Carlitz,
Ehrenborg_Readdy_juggling,
Gould, 
Milne_q_analog,
Wachs_White,White}.

In the literature there are 
many identities involving Stirling and
$q$-Stirling numbers of the second kind.
Stirling identities which appear in
Jordan's text~\cite{Jordan}
have been transformed to $q$-identities
by Ernst~\cite{Ernst}
using the theory of Hahn--Cigler--Carlitz--Johnson, Carlitz--Gould
and the Jackson $q$-derivative.
Verde-Star uses the divided difference operator~\cite{Verde-Star}
and the complete homogeneous symmetric polynomials
in the indeterminates
$x_k = 1 + q + \cdots + q^{k}$
in his work~\cite{Verde-Star_transformation}.

The goal of this paper is to give bijective proofs
of many of these  $q$-Stirling identities 
as well as a number of new identities.
Underlying these proofs is the 
theory of restricted growth words
which we review in the next section.
In 
Section~\ref{section_recurrence}
we discuss recurrence structured $q$-Stirling identities,
while in
Section~\ref{section_generating_function}
we focus on Gould's ordinary generating function
for the $q$-Stirling number.
A poset theoretic proof of Carlitz's identity
is given in Section~\ref{section_poset_Carlitz}.
Section~\ref{section_convolution} contains
combinatorial proofs of the de M\'edicis--Leroux 
$q$-Vandermonde convolutions.
We provide new proofs of Garsia and Remmel's
$q$-analogue of the Frobenius identity 
and
of Ehrenborg's Hankel $q$-Stirling
determinantal identity
in
Sections~\ref{section_Frobenius}
and~\ref{section_determinantal}.
In Section~\ref{section_another_identity_from_Carlitz}
we prove two identities of Carlitz
each using a sign-reversing involution on $RG$-words.
In the last section, 
we prove two parameter
$q$-Stirling identities, generalizing identities
of Mercier and include a symmetric function
reformulation.

\section{Preliminaries}
\label{section_preliminaries}

A word $w = w_{1} w_{2} \cdots w_{n}$ of length $n$ where
the entries are positive integers is called 
a {\em restricted growth word}, or 
{\em $RG$-word} for short, if
$w_{i}$ is bounded above by
$\max(0, w_{1}, w_{2}, \ldots, w_{i-1}) + 1$
for all indices $i$. 
This class of words was introduced by
Milne in the papers~\cite{Milne_restricted,Milne_q_analog}.
The set of $RG$-words of length $n$ where the largest entry is $k$ is
in bijective correspondence with set partitions of the 
set $\{1,2, \ldots, n\}$ into $k$ blocks.
Namely, if $w_{i} = w_{j}$, place the elements
$i$ and $j$ in the same
block of the partition.
To describe the inverse of this bijection,
write the partition
$\pi = B_{1} / B_{2} / \cdots / B_{k}$
in standard form, that is,
with $\min(B_{1}) < \min(B_{2}) < \cdots < \min(B_{k})$.
The associated $RG$-word is given by
$w = w_1 \cdots w_n$ where
$w_{i} = j$ if the entry $i$ 
appears in the $j$th block $B_{j}$ of $\pi$.

One way to obtain a $q$-analogue of Stirling numbers
of the second kind is to introduce a weight on
$RG$-words.
Let $\rg{n}{k}$ denote the set of all $RG$-words of length $n$ 
with maximal entry~$k$.
Observe that
$\rg{n}{k}$ is the empty set if $n < k$.
The set $\rg{n}{0}$ is also empty for $n > 0$
but the set $\rg{0}{0}$ is the singleton set
consisting of the empty word $\epsilon$.
Define the weight of $w = w_{1} w_{2} \cdots w_{n} \in \rg{n}{k}$
by
\begin{equation} 
    \wt(w) = q^{\sum_{i=1}^{n} (w_{i} - 1) - \binom{k}{2}}  . 
\end{equation} 
The {\em $q$-Stirling numbers of the second kind} are given
by
\begin{equation}
     S_{q}[n,k] = \sum_{w \in \rg{n}{k}} \wt(w).
\label{equation_rg_interpretation}
\end{equation}
See Cai and Readdy~\cite[Sections~2 and~3]{Cai_Readdy}.

This definition satisfies the
recurrence definition for $q$-Stirling numbers of the second kind
originally due to Carlitz; 
see~\cite[pages~128--129]{Carlitz} and~\cite[Section~3]{Carlitz_identity}:
\begin{equation*}
\label{equation_Stirling2_recurrence}
S_{q}[n,k] = S_{q}[n-1, k-1] + [k]_{q} \cdot S_{q}[n-1,k] 
           \:\: \mbox{ for } 1 \leq k \leq n,
\end{equation*}
where
$[k]_q = 1 + q + \cdots + q^{k-1}$.
To see this,
consider a word $w \in \rg{n}{k}$.
If the last letter is a left-to-right maxima then
the word $w$ is of the form $w = v k$ where
$v \in \rg{n-1}{k-1}$, yielding the first term of the recurrence.
Otherwise $w$ is of the form $w = v i$ where
$v \in \rg{n-1}{k}$ and $1 \leq i \leq k$,
which yields the second term of the recurrence.
The boundary conditions
$S_{q}[n,0] = \delta_{n,0}$ and $S_{q}[0,k] = \delta_{0,k}$
also follow from
the interpretation~\eqref{equation_rg_interpretation}.
For other weightings of $RG$-words which generate
the $q$-Stirling numbers of the second kind,
see~\cite{Milne_restricted_rank_row}
and~\cite{Wachs_White}.

For a word $w = w_{1} w_{2} \cdots w_{n}$
define the {\em length} of $w$ to be $\ell(w) = n$.
Similarly, define the {\em $\ls$-weight} of $w$
to be 
$\ls(w) = q^{\sum_{i=1}^{n} (w_{i}-1)}$.
This is a generalization of the
$ls$-statistic of $RG$-words~\cite[Section~2]{Wachs_White}.
The concatenation of two words $u$ and $v$ 
is denoted by $u \cdot v$.
The word~$v$ is a {\em factor} of the word $w$ if
one can write $w = u_{1} \cdot v \cdot v_{2}$.
A word
$v = v_{1} v_{2} \cdots v_{k}$ is a {\em subword} of~$w$ if there
is a sequence $1 \leq i_{1} < i_{2} < \cdots < i_{k} \leq n$
such that $w_{i_{j}} = v_{j}$ for all $1 \leq j \leq k$.
In other words, a factor of $w$ is a subword consisting of
consecutive entries.
For $S$ a set of positive integers, let $S^{k}$ denote
the set of all words of length $k$ with
entries in $S$.
Furthermore, let $S^{*}$ denote
the union
$S^{*} = \bigcupdotalt_{k \geq 0} S^{k}$,
that is,
the set of all words with
entries in $S$.
Observe that when $S$ is the empty set
then $S^{*}$ consists only of the empty word $\epsilon$.
Let $[j,k]$ denote the interval
$[j,k] = \{i \in \Ppp \: : \: j \leq i \leq k\}$.

Recall the $q$-Stirling numbers of the second kind
are specializations of
the homogeneous symmetric function $h_{n-k}$:
\begin{equation}
S_{q}[n,k] = h_{n-k}([1]_{q}, [2]_{q}, \ldots, [k]_{q}).
\label{equation_complete_symmetric_functions}
\end{equation}
See for instance~\cite[Chapter~I, Section~2, Example~11]{Macdonald}.
This follows directly by observing that a
word $w \in \rg{n}{k}$ has a unique expansion of the
form
\begin{equation}
w
=
1 \cdot u_{1} \cdot 2 \cdot u_{2} \cdots k \cdot u_{k} ,
\label{equation_expansion}
\end{equation}
where $u_{i}$ is a word in $[1,i]^{*}$.
By summing
over all words $u_{i}$ for
$i = 1, \ldots, k$
such that the sum of their lengths is
$\ell(u_{1}) + \ell(u_{2}) \plusdots \ell(u_{k}) = n-k$,
equation~\eqref{equation_complete_symmetric_functions}
follows.

\section{Recurrence related identities}
\label{section_recurrence}

In this section we focus on recurrence structured identities
for the
 $q$-Stirling numbers of the second kind.
The proofs we provide here use
the combinatorics of $RG$-words.

We begin with Mercier's identity~\cite[Theorem~3]{Mercier_q_identities}.
This is a $q$-analogue of Jordan~\cite[equation~9, page~187]{Jordan}. 
Mercier's original proof of
Theorem~\ref{theorem_Mercier_recurrence} 
was by induction.
Later a combinatorial proof using $0$-$1$ tableaux was given by 
de M\'edicis and Leroux~\cite{de_Medicis_Leroux_unified}. 
In the same paper, de M\'edicis and Leroux proved 
Theorems~\ref{theorem_de_Medicis_Leroux_recurrence} 
and \ref{theorem_de_Medicis_Leroux_recurrence_2} using 
$0$-$1$ tableaux.

\begin{theorem}[Mercier, 1990]
\label{theorem_Mercier_recurrence}
For nonnegative integers $n$ and $k$,
the following identity holds:
\begin{equation}
     S_{q}[n+1, k+1] = \sum_{m = k}^{n} \binom{n}{m} \cdot q^{m-k} 
                                    \cdot S_{q}[m,k]  .
\end{equation}
\end{theorem}
\begin{proof}
When $n < k$ there is nothing to prove.
For any word $w \in \rg{n+1}{k+1}$, 
suppose there are $m$ entries in $w$ that are not equal to one. 
Remove the $n+1-m$ entries equal to one in~$w$ 
and then subtract one from each of the remaining~$m$ entries 
to obtain a new word $u$. 
Observe $u \in \rg{m}{k}$ and
$\wt(w) = q^{m-k} \cdot \wt(u)$. 
Conversely, given a word $u \in \rg{m}{k}$,
one can first increase each of the $m$ entries by one
and then insert $n+1-m$ ones into the word to obtain
an $RG$-word $w \in \rg{n+1}{k+1}$.
There are $\binom{n}{n-m} = \binom{n}{m}$ ways
to insert the~$n+1-m$ ones since the first entry
in an $RG$-word must be one.
In other words, for any $u \in \rg{m}{k}$
we can obtain $\binom{n}{m}$ new $RG$-words in $\rg{n+1}{k+1}$
under the map described above, which gives the desired identity.
\end{proof}

Using similar ideas we also prove the following $q$-identity.
It is a $q$-analogue of a result  due to
Jordan~\cite[equation~7, page~187]{Jordan}.

\begin{theorem}
\label{theorem_q_Jordan}
For two non-negative integers $n$ and $m$,
the following identity holds:
\begin{equation}
q^{n-m} \cdot S_{q}[n,m]
=
\sum_{k=m}^{n}
(-1)^{n-k} \cdot \binom{n}{k} \cdot S_{q}[k+1,m+1] .
\label{equation_alternating_sum}
\end{equation}
\end{theorem}
\begin{proof}
For a subset $A \subseteq \{2,3, \ldots, n+1\}$
observe that the sum over the weights of
$RG$-words in the set $\rg{n+1}{m+1}$
with ones in the set of positions containing the set $A$
is given by $S_{q}[n+1-|A|,m+1]$. 
Hence by inclusion-exclusion the right-hand side
of equation~\eqref{equation_alternating_sum}
is the sum of the weights of all words in
$\rg{n+1}{m+1}$ 
where the element $1$ only occurs in
first position.
This set of $RG$-words is also obtained by taking
a word in $\rg{n}{m}$, adding one to each entry,
which multiplies the weight by~$q^{n-m}$,
and concatenating it with a one on the left.
\end{proof}

Theorems~\ref{theorem_de_Medicis_Leroux_recurrence}
and
\ref{theorem_de_Medicis_Leroux_recurrence_2}
appear 
in~\cite[Propositions~2.3 and~2.5]{de_Medicis_Leroux_unified}.
We now give straightforward proofs of each result using
$RG$-words.
\begin{theorem}[de M\'edicis--Leroux, 1993]
\label{theorem_de_Medicis_Leroux_recurrence}
For nonnegative integers $n$ and $k$,
the following identity holds:
\begin{equation}
S_{q}[n+1, k+1] = \sum_{j = k} ^{n} [k+1]_{q}^{n-j} \cdot S_{q}[j,k].
\end{equation}
\end{theorem}
\begin{proof}
Factor a word $w \in \rg{n+1}{k+1}$
as $w = x \cdot (k+1) \cdot y$
where $x \in \rg{j}{k}$ for some $j \geq k$
and $y$ belongs to $[1,k+1]^{*}$.
The factor $y$ has length $n-j$.
The sum of the weights of these words is
$[k+1]_{q}^{n-j} \cdot S_{q}[j,k]$.
The result follows by summing over all possible integers $j$.
\end{proof}

\begin{theorem}[de M\'edicis--Leroux, 1993]
\label{theorem_de_Medicis_Leroux_recurrence_2}
For nonnegative integers $n$ and $k$,
the following identity holds:
\begin{equation}
     (n-k)\cdot S_{q}[n,k] = \sum_{j = 1}^{n-k} 
                            S_{q}[n-j, k] \cdot 
                   ([1]_{q}^{j} + [2]_{q}^{j} + \cdots + [k]_{q}^{j}).
\end{equation}
\end{theorem}

\begin{proof}
For a  word $w \in \rg{n}{k}$ consider
factorizations $w = x \cdot y \cdot z$
with the following two properties:
(1)
the rightmost letter of the factor $x$,
call this letter~$i$, 
 is a left-to-right maxima of $x$,
and
(2) the word $y$ is non-empty and all letters of $y$ are at most~$i$. 

We claim that the number of such factorizations of $w$ is $n-k$.
Let $s_{i}$ be the number of letters between the first occurrence
of $i$ and the first occurrence of $i+1$, and let $s_{k}$ be the
number of letters after the first occurrence of $k$.
For a particular $i$, we have $s_{i}$ choices for the word $y$.
But $\sum_{i=1}^{k} s_{i} = n-k$
since there are $n-k$ repeated letters in $w$.
This completes the claim.

Fix integers $1 \leq j \leq n-k$ and $1 \leq i \leq k$.
Given a word $u \in \rg{n-j}{k}$, we can factor it
uniquely as~$x \cdot z$, where the last letter of $x$
is the first occurrence of $i$ in the word $u$.
Pick $y$ to be any word of length $j$ with letters at most $i$.
Finally, let $w = x \cdot y \cdot z$.
Observe that this is a factorization satisfying
the conditions from the previous paragraph. 
Furthermore, we have $\wt(w) = \wt(u) \cdot \ls(y)$.
Summing over all words $u \in \rg{n}{k}$ 
and words $y \in [1,i]^{j}$
yields $S_{q}[n-j, k] \cdot [i]_{q}^{j}$.
Lastly, summing over all $i$ and $j$ gives the desired equality.
\end{proof}

\section{Gould's generating function}
\label{section_generating_function}

Gould~\cite[equation~(3.4)]{Gould} gave an analytic proof
for the ordinary generating function of the
$q$-Stirling numbers of the second kind.
Later Ernst~\cite[Theorem~3.22]{Ernst} gave a proof using
the orthogonality of the $q$-Stirling numbers of the first
and second kinds.
Wachs and White~\cite{Wachs_White} stated a 
$p,q$-version of this generating function without proof.
Here we
prove Gould's $q$-generating function using $RG$-words.

\begin{theorem}[Gould, 1961]
The $q$-Stirling numbers of the second kind $S_{q}[n,k]$ have 
the generating function
\begin{equation}
\label{equation_generating_function}
     \sum_{n \geq k} S_{q}[n,k] \cdot t^{n} 
     = \frac{t^k}{\prod_{i=1}^k (1-[i]_{q} \cdot t)}.
\end{equation}
\end{theorem}
\begin{proof}
The left-hand side 
of~\eqref{equation_generating_function}
is the sum of over all $RG$-words $w$ of length
at least $k$ with largest letter $k$ where each term is
$\wt(w) \cdot t^{\ell(w)}$.
Using the expansion
in equation~\eqref{equation_expansion},
that is,
$w = 1 \cdot u_{1} \cdot 2 \cdot u_{2} \cdots k \cdot u_{k}$
where $u_{i}$ is a word in $[1,i]^{*}$,
observe the weight
of $w$ factors as
$\wt(w) = \ls(u_{1}) \cdot \ls(u_{2}) \cdots \ls(u_{k})$
whereas the term
$t^{\ell(w)}$ factors
as
$t^{k} \cdot t^{\ell(u_{1})} \cdot t^{\ell(u_{2})} \cdots t^{\ell(u_{k})}$.
Since there are no restrictions on the length of $i$th word $u_{i}$,
all the words $u_{i}$ for $i = 1, \ldots, k$
together contribute the factor
$1 + [i]_{q} \cdot t + [i]_{q}^{2} \cdot t^{2} + \cdots
= 1/(1 -  [i]_{q} \cdot t)$.
By multiplying together all the contributions from
all of the words $u_{i}$,
the identity follows.
\end{proof}

\section{A poset proof of Carlitz's identity}
\label{section_poset_Carlitz}

In this section we state a poset
decomposition theorem for the Cartesian product
of chains.
This decomposition implies 
Carlitz's identity.
For basic poset terminology 
and background,
we refer the reader to Stanley's treatise~\cite[Chapter~3]{Stanley_EC_I}.

Let $C_{m}$ denote 
the chain on $m$ elements.
Recall that $\Ppp^{n}$ is the set of all words
of length~$n$ having positive integer entries.
We make this set into a poset, in fact, a lattice,
by entrywise comparison, where
the partial order relation is given by
$v_{1} v_{2} \cdots v_{n} \leq w_{1} w_{2} \cdots w_{n}$
if and only if $v_{i} \leq w_{i}$ for all indices $1 \leq i \leq n$.
Note that $[1,m]^{n}$ is the subposet consisting of all
words of length~$n$
where the entries are at most~$m$.

For a word $v$ in $\rg{n}{k}$,
factor $v$ according to equation~\eqref{equation_expansion},
that is,
write $v$ as the product
$v = 1 \cdot u_{1} \cdot 2 \cdot u_{2} \cdots u_{k-1} \cdot k \cdot u_{k}$,
where each factor~$u_i$ belongs to $[1,i]^*$.
For $m \geq n$ define the word 
$\omega_{m}(v) = 
m \cdot u_{1} \cdot m \cdot u_{2} \cdots u_{k-1} \cdot m \cdot u_{k}$.
Effectively,  each left-to-right maxima is replaced by 
the integer $m$.
Directly it is clear that the interval
$[v,\omega_{m}(v)]$ in $\Ppp^{n}$
is isomorphic to a product of chains, that is,
$$
[v,\omega_{m}(v)] 
\cong
C_{m} \times C_{m-1} \timesdots C_{m-k+1} .
$$

\begin{theorem}
The $n$-fold Cartesian product of the $m$-chain
has the decomposition
\begin{equation*}
[1,m]^{n}
=
\bigcupdotalt_{0 \leq k \leq \min(m,n)} \:  
\bigcupdotalt_{v \in \rg{n}{k}}
[v,\omega_{m}(v)] .
\end{equation*}
\label{theorem_Carlitz_poset}
\end{theorem}
\begin{proof}
Define a map $f : \Ppp^{n} \longrightarrow \Ppp^{n}$ as follows.
Let $w = w_{1} w_{2} \cdots w_{n}$ be a word in $\Ppp^{n}$.
If $w$ is an $RG$-word, let $f(w) = w$.
Otherwise let $i$ be the smallest index in
$w$ reading from left to right
that makes~$w$ fail to be an $RG$-word.
In other words, $i$ is the smallest index such that
$\max(0,w_{1}, w_{2}, \ldots, w_{i-1})+1 < w_{i}$.
Let $f(w)$ be the new word formed by
replacing the $i$th entry of $w$
with
$\max(0,w_{1}, w_{2}, \ldots, w_{i-1})+1$.
Observe that for all words $w$
we obtain the poset inequality
$f(w) \leq w$.

Since the word $w$ only has $n$ entries, we know that
the $(n+1)$st iteration of $f$ is equal to the
$n$th iteration of $f$, that is, $f^{n+1}(w) = f^{n}(w)$.
Furthermore, $f^{n}(w)$ is an $RG$-word.
Finally, define
$\varphi : \Ppp^{n} \longrightarrow 
\bigcupdotalt_{0 \leq k \leq n} \rg{n}{k}$
to be the map $f^{n}$.
Observe that $\varphi$ is a surjection since every $RG$-word
is a fixed point.
Furthermore, for all words $w$ the
inequality  $\varphi(w) \leq w$ holds
in the poset $\Ppp^{n}$.

Let $v$ be a word in $\rg{n}{k}$.
Use the expansion~\eqref{equation_expansion}
to write $v$ in the form
$v = 1 \cdot u_{1} \cdot 2 \cdot u_{2} \cdots u_{k-1} \cdot k \cdot u_{k}$,
where $u_{i} \in [1,i]^{*}$.
It is straightforward to check that the
fiber $\varphi^{-1}(v)$ is given by
\begin{equation}
\varphi^{-1}(v)
     =
\{ j_{1} \cdot u_{1} \cdot j_{2} \cdot u_{2} \cdots 
           u_{k-1} \cdot j_{k} \cdot u_{k}
                 \: : \: 
       i \leq j_{i} \text{ for } i=1, 2, \ldots, k\}.
\label{equation_fiber}
\end{equation}
Observe that 
as a poset this fiber is isomorphic to $\Ppp^{k}$.
When we restrict to $[1,m]^{n}$
we obtain that the intersection
$\varphi^{-1}(v) \cap [1,m]^{n}$ 
is the interval $[v,\omega_{m}(v)]$.
Taking the disjoint union over all
$RG$-words~$v$, the decomposition follows.
\end{proof}

By considering the rank generating function of
Theorem~\ref{theorem_Carlitz_poset},
we can obtain a poset
theoretic proof of Carlitz's identity~\cite[Section~3]{Carlitz_identity}.
Other proofs are due to Milne
using finite operator techniques on 
restricted growth functions~\cite{Milne_q_analog},
de M\'edicis and Leroux via interpreting the identity
as counting products of matrices over the finite field $\GF(q)$
having non-zero columns~\cite{de_Medicis_Leroux_unified},
and Ehrenborg and Readdy using the theory of juggling
sequences~\cite{Ehrenborg_Readdy_juggling}.
\begin{corollary}[Carlitz, 1948]
The following $q$-identity holds:
\begin{equation}
\label{equation_Carlitz_identity}
     [m]_{q}^{n}
   =
   \sum_{k=0}^{n}
     q^{\binom{k}{2}} \cdot S_{q}[n, k] \cdot [k]_{q}! \cdot \qchoose{m}{k}{q}.
\end{equation}
\label{corollary_Carlitz_identity}
\end{corollary}
\begin{proof}
The cases when $n=0$ or $m=0$ are straightforward.
The rank generating function of
 the left-hand side of Theorem~\ref{theorem_Carlitz_poset}
is $[m]_{q}^{n}$.
The rank generating function of
the interval
$[v, \omega_{m}(v)]$ is
$q^{\binom{k}{2}} \cdot \wt(v)
\cdot [m]_{q} \cdot [m-1]_{q} \cdots [m-k+1]_{q}$.
By summing over all $RG$-words,
the result follows.
\end{proof}

\begin{remark}
{\em
Similar 
poset techniques used to prove 
Theorem~\ref{theorem_Carlitz_poset}
and 
Corollary~\ref{corollary_Carlitz_identity}  
can be applied to obtain the identities in 
Section~\ref{section_recurrence}.
The proofs are omitted.}
\end{remark}

The map $\varphi$ that appears in the proof of
Theorem~\ref{theorem_Carlitz_poset} has interesting
properties.
\begin{proposition}
The map
$\varphi : \Ppp^{n} \longrightarrow
\bigcupdotalt_{0 \leq k \leq n} \rg{n}{k}$
is the dual of a closure operator, that is, it satisfies the following
three properties:
\vspace*{-3mm}
\begin{enumerate}[(i)]
\item
$\varphi(w) \leq w$,
\vspace*{-2.5mm}
\item
$\varphi^{2}(w) = \varphi(w)$ and
\vspace*{-2.5mm}
\item
$v \leq w$ implies that $\varphi(v) \leq \varphi(w)$.
\end{enumerate}
\end{proposition}
\begin{proof}
Properties~(i) and~(ii) are direct from the construction
of the map $\varphi$.
To prove property~(iii)
assume that
we have $v \leq w$
but
$\varphi(v) \not\leq \varphi(w)$.
Let $i$ be the smallest index such that
$\varphi(w)_{i} < \varphi(v)_{i}$.
Especially for $j < i$ we have
$\varphi(w)_{j} \geq \varphi(v)_{j}$.
Since 
$\varphi(w)_{i} < \varphi(v)_{i} \leq v_{i} \leq w_{i}$,
we know that the $i$th coordinate is changed
when computing $\varphi(w)$.
Hence
$\varphi(w)_{i}
=
\max(0,\varphi(w)_{1}, \ldots, \varphi(w)_{i-1}) + 1
\geq
\max(0,\varphi(v)_{1}, \ldots, \varphi(v)_{i-1}) + 1
\geq
\varphi(v)_{i}$,
a contradiction.
\end{proof}

As a consequence of the closure property,
if $v$ and $w$ are $RG$-words then so is
their join $v \join w = u$
where the $i$th entry is given by
$u_{i} = \max(v_{i},w_{i})$.
This can also be proven directly.
Note however that the set of $RG$-words
is not closed under the meet operation;
see for instance the two $RG$-words
$1123$ and $1213$.

\section{$q$-Vandermonde convolutions}
\label{section_convolution}

Verde-Star gave Vandermonde convolution identities for Stirling
numbers of the second kind~\cite[equations~(6.24), (6.25)]{Verde-Star}. 
Chen gave a grammatical
proof for the first of these identities~\cite[Proposition~4.1]{Chen}. 
For $q$-analogues of both identities, de M\'edicis and Leroux used
$0,1$-tableaux for their
argument~\cite[equations~(1.12), (1.14)]
{de_Medicis_Leroux_convolution}. 
In this section  we present 
combinatorial proofs of the de M\'edicis--Leroux results using $RG$-words.

As a remark, 
Theorem~\ref{theorem_Mercier_recurrence}
is the special case of $n=1$ in 
Theorem~\ref{theorem_Anne_and_Pierre_1}.

\begin{theorem}[de M\'edicis--Leroux, 1995]
The following $q$-Vandermonde convolution
holds for $q$-Stirling numbers of the second kind:
\begin{equation}
\label{equation_convolution_1}
     S_{q}[m+n,k] = \sum_{i+j \geq k} \binom{m}{j} 
                  \cdot q^{i\cdot (i+j-k)} \cdot [i]_{q}^{m-j} 
                  \cdot S_{q}[n,i] \cdot S_{q}[j, k-i] .
\end{equation}
\label{theorem_Anne_and_Pierre_1}
\end{theorem}
\begin{proof}
Given a  word $w \in \rg{m+n}{k}$, factor
it as~$w = u \cdot z$ where $u$ has length $n$ and
$i$ is the largest entry in $u$.
By assumption, $u \in \rg{n}{i}$.

The second factor $z = w_{n+1} \cdot w_{n+2} \cdots w_{n+m}$
has length $m$ and its maximal entry is at most~$k$. In particular, if
$i < k$ then the maximal entry for $z$ is exactly $k$. Suppose
there are $j$ entries in~$z$ that are strictly larger than
$i$. These $j$ entries from $z$ form a subword $v$.
Denote by $v^{(-i)}$ the shift of $v$ by
subtracting $i$ from each entry in $v$. It is straightforward to check
that $v^{(-i)} \in \rg{j}{k-i}$.

For any word $w \in \rg{m+n}{k}$, we can decompose it as
described above. In such a decomposition, the first segment $u$
contributes to a factor of $S_{q}[n, i]$. The subsequence $v$ of the second segment~$z$
contributes to a factor of $S_{q}[j, k-i] \cdot q^{i \cdot (j-(k-i))}$
since the shift $v^{(-i)}$ causes a weight loss of $q^{i}$ from
each of the 
$j-(k-i)$ repeated entries in $v$. Finally, the remaining 
entries in~$z$ that are less than or equal to $i$ range from $1$ to $i$.  Each will
contribute to a factor of $[i]_{q}$. These $m-j$ entries can be assigned
at any position in $z$, which gives~$\binom{m}{j}$
choices. Multiplying all these weights, we obtain the desired identity.
\end{proof}

Note that Theorem~\ref{theorem_de_Medicis_Leroux_recurrence}
is a special case of 
Theorem~\ref{theorem_Anne_and_Pierre_2}
when one takes $r=0$.

\begin{theorem}[de M\'edicis--Leroux, 1995]
The following $q$-Vandermonde convolution
holds for $q$-Stirling numbers of the second kind:
\begin{equation}
\label{equation_convolution_2}
     S_{q}[n+1, k+r+1] = \sum_{i=0}^{n} \sum_{j=r}^{i} \binom{i}{j} 
                       \cdot q^{(k+1) \cdot (j-r)} \cdot [k+1]_{q} ^{i-j} 
                       \cdot S_{q}[j,r] \cdot S_{q}[n-i, k].
\end{equation}
\label{theorem_Anne_and_Pierre_2}
\end{theorem}
\begin{proof}
This result is proved in a similar fashion
as Theorem~\ref{theorem_Anne_and_Pierre_1}.
For any $w \in \rg{n+1}{k+r+1}$, 
suppose $w$ is of the form $x \cdot (k+1) \cdot y$
where $x \in \rg{n-i}{k}$ for some $i$.
Consider the remaining word $y = w_{n-i+2} \cdots w_{n+1}$
of length $i$. The maximal entry of~$y$ is $k+r+1$. 
Suppose there are $j$ entries in~$y$ that are at least $k+2$.
These~$j$ entries form a subword $v$, and $v^{(-k-1)}$,
obtained by subtracting $k+1$ from each entry in~$v$,
is an $RG$-word in $\rg{j}{r}$, giving a total weight of $S_{q}[j,r]$.
The weight loss from the shift is
$q^{(k+1)\cdot (j-r)}$ since there are $j-r$ repeated entries.
The remaining $i-j$
entries in~$y$ can be any value from the interval~$[1,k+1]$.
Each such 
entry contributes to a factor of~$[k+1]_{q}$.  
Finally, there  are $\binom{i}{j}$ ways to place the $j$ entries
back into $u$.
This proves identity~\eqref{equation_convolution_2}.
\end{proof}

\begin{remark}
{\em
Theorem~\ref{theorem_q_Jordan}
can be viewed as an inversion of
Theorem~\ref{theorem_Mercier_recurrence}.
Furthermore,
Theorem~\ref{theorem_Mercier_recurrence}
is a special case of
Theorem~\ref{theorem_Anne_and_Pierre_1}.
Is there any sort of natural inversion 
analogue to
Theorem~\ref{theorem_Anne_and_Pierre_1}?
}
\end{remark}

\section{A $q$-analogue of the Frobenius identity}
\label{section_Frobenius}

We now prove a
$q$-analogue of the Frobenius identity by Garsia and 
Remmel~\cite[equation~I.1]{Garsia_Remmel}.

\begin{theorem}[Garsia--Remmel, 1986]
\label{theorem_Frobenius}
The following $q$-Frobenius identity holds:
\begin{equation}
\sum_{m \geq 0} [m]_{q}^{n} \cdot x^{m}
 = 
	\sum_{k=0}^{n} \frac{q^{\binom{k}{2}} 
	\cdot S_{q}[n,k] \cdot [k]_{q}! \cdot x^{k}}
	           {(1-x) \cdot (1-qx) \cdots (1-q^{k}x)}.
\label{equation_Frobenius}
\end{equation}
\end{theorem}
\begin{proof}
When $n=0$ the result is direct.
We concentrate on the case $n > 0$.
For a word $w$ in~$\Ppp^{n}$ let $\max(w)$ denote its
maximal entry.
Hence the left-hand side of
equation~\eqref{equation_Frobenius} 
is given by
$$
\sum_{m \geq 0} [m]_{q}^{n} \cdot x^{m}
 = 
\sum_{m \geq 0}
\sum_{\substack{w \in \Ppp^{n} \\ \max(w) \leq m}}
\ls(w) \cdot x^{m}   . $$
Recall the poset map
$\varphi : \Ppp^{n} \longrightarrow 
\bigcupdotalt_{0 \leq k \leq n} \rg{n}{k}$
appearing in the proof of
Theorem~\ref{theorem_Carlitz_poset}.
Let $v \in \rg{n}{k}$.
The fiber $\varphi^{-1}(v)$ is given
in equation~\eqref{equation_fiber}.
The sum over this fiber 
appears in the proof of Corollary~\ref{corollary_Carlitz_identity},
that is,
\begin{align*}
\sum_{\substack{w \in \varphi^{-1}(v) \\ \max(w) \leq m}}
\ls(w)
& =
\wt(v)
\cdot
q^{\binom{k}{2}}
\cdot
[m]_{q} \cdot [m-1]_{q} \cdots [m-k+1]_{q} .
\end{align*}
Multiplying the above by $x^{m}$ and summing over all $m \geq 0$ yields
\begin{align*}
\sum_{m \geq 0}
\sum_{\substack{w \in \varphi^{-1}(v) \\ \max(w) \leq m}}
\ls(w) \cdot x^{m}
& =
\wt(v)
\cdot
q^{\binom{k}{2}}
\cdot
[k]_{q}!
\cdot
\sum_{m \geq 0}
\qchoose{m}{k}{q}
\cdot x^{m} \\
& =
\frac{
\wt(v)
\cdot
q^{\binom{k}{2}}
\cdot
[k]_{q}!
\cdot
x^{k}}{(1-x) \cdot (1-qx) \cdots (1-q^{k}x)} .
\end{align*}
The result now follows
by summing over all $RG$-words of length $n$.
\end{proof}

\section{A determinantal identity} 
\label{section_determinantal}

The following identity was first stated by
Ehrenborg~\cite[Theorem~3.1]{Ehrenborg_determinants}
who proved it using juggling patterns.
We now present a proof using $RG$-words.

\begin{theorem}[Ehrenborg, 2003]
\label{theorem_determinants}
Let $n$ and $s$ be non-negative integers. Then the following identity holds:
$$
\det(S_{q}[s+i+j, s+j])_{0 \leq i,j \leq n} = [s]_{q}^0 \cdot [s+1]_{q}^1 \cdots [s+n]_{q}^{n}.
$$
\end{theorem}

\begin{proof}
Let $T$ be the set of all $(n+2)$-tuples
$(\sigma, w(0), w(1), \ldots, w(n))$
where $\sigma$ is a permutation of
the $n+1$ elements $\{0,1, \ldots, n\}$,
and $w(i)$ is a word in
$\rg{s+i+\sigma(i)}{s+\sigma(i)}$
for all $0 \leq i \leq n$.
The determinant expands as the sum
$$
\det(S_{q}[s+i+j, s+j])_{0 \leq i,j \leq n}
=
\sum_{(\sigma, w(0), \ldots, w(n)) \in T}
(-1)^{\sigma} \cdot \wt(w(0)) \cdot \wt(w(1)) \cdots \wt(w(n)).
$$
Factor the word $w(i) \in \rg{s+i+\sigma(i)}{s+\sigma(i)}$
as $w(i) = u(i) \cdot v(i)$
where the lengths are given by
$\ell(u(i)) = s+i$
and
$\ell(v(i)) = \sigma(i)$.
Furthermore,
let $a_{i}$ denote
the number of repeated entries in the $RG$-word $w(i)$
that appear in the factor $v(i)$,
that is,
$$
a_{i} = |\{j \: : \: j > s+i, 
     w(i)_{j} = w(i)_{r} \text{ for some } r < j\}|.
$$
There are $\sigma(i) - a_{i}$
left-to-right maxima
of $w(i)$ that appear in $v(i)$.
Since
$w(i)$ has $s+\sigma(i)$ 
left-to-right maxima, we obtain that the
first factor
$u(i)$
has
$(s+\sigma(i)) - (\sigma(i) - a_{i}) = s+a_{i}$
left-to-right maxima,
that is, the factor $u(i)$
belongs to the set
$\rg{s+i}{s+a_{i}}$.
To be explicit,
the left-to-right maxima of $u(i)$
are given by
$1, 2, \ldots, s+a_{i}$.
Lastly, observe that there are $i$ repeated entries in any word
$w(i) \in \rg{s+i+\sigma(i)}{s+\sigma(i)}$
and $\sigma(i)$ is the length of $v(i)$,
yielding the bound $a_{i} \leq \min(i, \sigma(i))$ for all~$i$.

Let $T_{1} \subseteq T$ consist of all tuples
$(\sigma, w(0), \ldots, w(n))$ where the 
sequence of $a_{i}$'s for $i = 0, \ldots, n$ are distinct.
This implies $a_{i} = i = \sigma(i)$, that is,
$\sigma$ is the identity permutation.
Furthermore, the first factor $u(i)$ is equal to $12\cdots (s+i)$
and the second factor $v(i)$ 
can be any word of length $i$
with the entries from the interval $[1,s+i]$.
Thus $\wt(w(i)) = \ls(v(i))$ and the sum over all such words $v(i)$
gives a total weight of~$[s+i]_{q}^{i}$. Thus we have
\begin{equation}
\label{equation_determinant_subset1}
\sum_{(\sigma, w(0), \ldots, w(n)) \in T_{1}}
(-1)^{\sigma} \cdot \wt(w(0)) \cdot \wt(w(1)) \cdots \wt(w(n))
=
\prod_{i=0}^{n} [s+i]_{q}^{i}.
\end{equation}

Let $T_{2} = T-T_{1}$ be the complement of $T_{1}$.
Define a sign-reversing involution $\varphi$
on $T_{2}$ as follows.
For $t = (\sigma, w(0), \ldots, w(n)) \in T_{2}$
there exists indices $i_{1}$ and $i_{2}$
such that $a_{i_{1}} = a_{i_{2}}$.
Let $(j,k)$ be the least such pair of indices in the lexicographic order.
First let $\sigma' = \sigma \circ (j,k)$
where $(j,k)$ denotes the transposition.
Second, 
let $w(i)' = w(i)$ for $i \neq j,k$.
Finally, define
$w(j)'$ and $w(k)'$
by switching the second factors
in the factorizations, that is,
$w(j)' = u(j) v(k)$
and
$w(k)' = u(k) v(j)$.
Overall, the function is given by
$\varphi(t) = (\sigma', w(0)', \ldots, w(n)')$

Since $u(j)$ and $u(k)$ have the same number of
left-to-right maxima, it is straightforward to check
that 
$w(j)' = u(j) v(k)$ belongs to
$\rg{s+j+\sigma(k)}{s+\sigma(k)}
=
\rg{s+j+\sigma'(j)}{s+\sigma'(j)}$.
Hence it follows that $\varphi(t) \in T_{2}$.
Let $a_{i}'$ be the number repeated entries in
$w(i)'$ that occur beyond position $s+i$.
Directly, we have $a_{i}' = a_{i}$ and
we obtain that $\varphi$ is an involution.
Finally, we have $(-1)^{\sigma'} = -(-1)^{\sigma}$
implying $\varphi$ is a sign-reversing involution.

Finally, it is direct to see that
$\wt(w(j)) \cdot \wt(w(k)) = \wt(w(j)') \cdot \wt(w(k)')$
using the observation that
$u(j)$ and $u(k)$ have the same number of left-to-right maxima.
Hence the map $\varphi$ is a sign-reversing involution
on $T_{2}$
which preserves the weight
$\wt(w(0)) \cdots \wt(w(n))$.
Thus the determinant is given by
equation~\eqref{equation_determinant_subset1}.
\end{proof}

\section{A pair of identities of Carlitz}
\label{section_another_identity_from_Carlitz}

In this section we turn our attention to two
identities of Carlitz.  Observe that setting
$q=1$ in 
Theorems~\ref{theorem_Carlitz_funny_q_1}
and~\ref{theorem_Carlitz_funny_q_2}
in this section does not yield any information
about the Stirling number $S(n,k)$ of the second kind.

We first prove a theorem from which
Carlitz's Theorem~\ref{theorem_Carlitz_funny_q_1}
will follow.

\begin{theorem}
For two non-negative integers $n$ and $k$
not both equal to $0$,
the following identity holds:
\begin{equation}
(1-q)^{n-k} \cdot S_{q}[n,k]
=
\sum_{j=0}^{n-k}
(-q)^{j} \cdot \binom{n-1}{n-k-j} \cdot \qchoose{j+k-1}{j}{q} .
\label{equation_Stanton_Gould_Carlitz_preliminary}
\end{equation}
\label{theorem_Stanton_Gould_Carlitz_preliminary}
\end{theorem}
\begin{proof} 
For a word $u = u_{1} u_{2} \cdots u_{n}$ in $\rg{n}{k}$
let $\NLRM(u)$ be the set of all positions $r$ such that
the letter $u_{r}$ is not a left-to-right-maxima of the word $u$,
that is, $u_{r} \leq \max(u_{1}, u_{2}, \ldots, u_{r-1})$.
Furthermore, for a position $r \in NLRM(u)$
define the bound $b(r)$ to be
$\max(u_{1}, u_{2}, \ldots, u_{r-1})$.
Note that $b(r)$ is the largest possible value we could change
$u_{r}$ to be so that the resulting word remains an $RG$-word
in $\rg{n}{k}$.

To describe the left-hand side
of~\eqref{equation_Stanton_Gould_Carlitz_preliminary},
consider the set
of pairs $(u,P)$ where $u \in \rg{n}{k}$ and
$P \subseteq \NLRM(u)$.
Define the weight of such a pair $(u,P)$ to be
$(-q)^{|P|} \cdot \wt(u)$.
It is clear that the sum of the weight over all such pairs $(u,P)$
is given by the left-hand side
of~\eqref{equation_Stanton_Gould_Carlitz_preliminary}.

Define a sign-reversing involution as follows.
For the pair $(u,P)$ pick the smallest
position $r$ in $\NLRM(u)$ such that
either $r \in P$ and $u_{r} \leq b(r)-1$
or $r \not\in P$ and $2 \leq u_{r}$.
In the first case,
send~$P$ to $P-\{r\}$
and replace the $r$th letter $u_{r}$ with $u_{r}+1$.
In the second case,
send $P$ to $P \cup \{r\}$
and replace the $r$th letter $u_{r}$ with $u_{r}-1$.
This is a sign-reversing involution which
pairs terms having the same weight, but opposite signs.
Furthermore, the remaining pairs $(u,P)$
satisfy
for all positions $r \in \NLRM(u)$
either 
$r \in P$ and $u_{r} = b(r)$
or
$r \not\in P$ and $u_{r} = 1$.

We now sum the weight of these remaining pairs $(u,P)$.
First select the cardinality $j$ of the set~$P$.
Note that $0 \leq j \leq n-k$
and that it yields a factor of $(-q)^{j}$.
Second, select the positions $r$ of the non-left-right-maxima
such that $r$ will not be in the set $P$ and $u_{r} = 1$.
There will be $n-k-j$ such positions and they can be anywhere
in the interval $[2,n]$, yielding
$\binom{n-1}{n-k-j}$ possibilities.
Third, select a weakly increasing word $z$ of length $j$ with letters
from the set $[k]$. This will be the letters corresponding to
the non-left-to-right-maxima such that their positions belong to set $P$.
The $\ls$-weight of these letters will be
the $q$-binomial coefficient $\qchoose{j+k-1}{j}{q}$.
Fourth, insert into the word $z$ the letters of the left-to-right-maxima.
There is a unique way to do this insertion.
Finally, insert the~$1$'s corresponding to positions not in the set $P$,
which were already been chosen by the binomial coefficient.
\end{proof} 

Note that the above proof fails when $n=k=0$
since we are using that a non-empty $RG$-word must begin with
the letter $1$, whereas the empty $RG$-word does not.

The next identity is due to Carlitz~\cite[equation~(9)]{Carlitz}.
It was stated by Gould~\cite[equation~(3.10)]{Gould}.
As a warning to the reader,
Gould's notation $S_{2}(n,k)$
for the $q$-Stirling number of the second kind
is related to ours by
$S_{2}(n,k) = S_{q}[n+k,n]$.
This identity also appears
in
the paper of
de M\'edicis, Stanton and White~\cite[equation~(3.3)]{de_Medicis_Stanton_White}
using modern notation.

\begin{theorem}[Carlitz, 1933]
\label{theorem_Carlitz_funny_q_1}
For two non-negative integers $n$ and $k$
the following identity holds:
\begin{equation}
(1-q)^{n-k} \cdot S_{q}[n,k]
=
\sum_{j=0}^{n-k}
(-1)^{j} \cdot \binom{n}{k+j} \cdot \qchoose{j+k}{j}{q} .
\label{equation_Stanton_Gould_Carlitz}
\end{equation}
\end{theorem}
\begin{proof}
When $n=k=0$ the statement is direct.
It is enough to show that the right-hand sides of
equations~\eqref{equation_Stanton_Gould_Carlitz_preliminary}
and~\eqref{equation_Stanton_Gould_Carlitz} agree.
We have
\begin{align*}
\sum_{j=0}^{n-k}
& 
(-1)^{j} \cdot \binom{n-1}{n-k-j} \cdot q^{j} \cdot \qchoose{j+k-1}{j}{q} \\
& =
\sum_{j=0}^{n-k}
(-1)^{j} \cdot \binom{n-1}{n-k-j} \cdot
\left( \qchoose{j+k}{j}{q} - \qchoose{j+k-1}{j-1}{q} \right) \\
& =
\sum_{j=0}^{n-k}
(-1)^{j} \cdot \binom{n-1}{n-k-j} \cdot
\qchoose{j+k}{j}{q}
-
\sum_{j=1}^{n-k}
(-1)^{j} \cdot \binom{n-1}{n-k-j} \cdot
\qchoose{j+k-1}{j-1}{q} \\
& =
\sum_{j=0}^{n-k}
(-1)^{j} \cdot \binom{n-1}{n-k-j} \cdot
\qchoose{j+k}{j}{q}
-
\sum_{j=0}^{n-k-1}
(-1)^{j+1} \cdot \binom{n-1}{n-k-j-1} \cdot
\qchoose{j+k}{j}{q} \\
& =
\sum_{j=0}^{n-k}
(-1)^{j} \cdot \binom{n}{n-k-j} \cdot
\qchoose{j+k}{j}{q} .
\end{align*}
Here we used the Pascal recursion for the $q$-binomial coefficients
in the first step, shifted $j$ to $j+1$ in the second sum in the third
step, and applied the Pascal recursion for the binomial coefficients
in the last step.
\end{proof}

The next identity is also due to Carlitz; see~\cite[equation~(8)]{Carlitz}.
It is equivalent to the previous identity, but we provide a proof
using $RG$-words.
\begin{theorem}[Carlitz, 1933]
\label{theorem_Carlitz_funny_q_2}
For $n$ and $k$ two non-negative integers the following
identity holds:
$$
\qchoose{n}{k}{q}
=
\sum_{j=k}^{n}
(q-1)^{j-k} \cdot \binom{n}{j} \cdot S_{q}[j,k] .
$$
\end{theorem}
\begin{proof}
The right-hand side describes the following collection
of triplets
$(A, u, P)$
where $A$ is a subset of the set $[n]$,
$u$ is a word in
$\rg{|A|}{k}$
and
$P$ is a subset of $\NLRM(u)$.
Define the weight of the triple 
$(A, u, P)$
to be the product $(-1)^{j-k-|P|} \cdot q^{|P|} \cdot \wt(u)$.
To better visualize the pair~$(A,u)$,
define the word $w = w_{1} w_{2} \cdots w_{n}$ of length $n$
with the letters in the set $\{0\} \cup [k]$
as follows.
Write $A$ as the increasing set
$\{a_{1} < a_{2} < \cdots < a_{j}\}$
and let
$w_{a_{r}} = u_{r}$.
The remaining letters of $w$ are set to be~$0$,
that is, if $i \not\in A$ let $w_{i} = 0$.
Note that the word $w$ uniquely encodes the pair~$(A,u)$.
Let $Q$ be the set
$Q = \{a_{r} \: : \: r \in P\}$,
that is, the set $Q$ encodes the subset $P$, where these
non-left-to-right-maxima occur in the longer word $w$.

Similar to the proof of
Theorem~\ref{theorem_Stanton_Gould_Carlitz_preliminary}
we define a sign-reversing involution by
selecting the smallest $r \in \NLRM(w)$
such that
$r \not\in Q$ and $w_{r} \geq 2$,
or
$r \in Q$ and $1 \leq w_{r} \leq b(r)-1$.
In the first case decrease $w_{r}$ by $1$
and join $r$ to the subset $Q$.
In the second case, increase $w_{r}$ by~$1$
and remove~$r$ from $Q$.
The remaining words $w$ satisfy the following:
for a non-left-to-right-maxima $r$
such that $w_{r} \geq 1$
either
$r \not\in Q$ and $w_{r} = 1$ both hold,
or
$r \in Q$ and $w_{r} = b(r)$ hold.

On the remaining pairs $(w,Q)$
define yet again a sign-reversing involution.
Let $i > a_{1}$ be the smallest index $i$
such that 
$w_{i} = 0$,
or
$w_{i} = 1$ and $i \not\in Q$.
This involution replaces $w_{i}$ with $1-w_{i}$.
Note that this involution changes the sign.

The pairs $(w,Q)$ which remain unmatched under
this second involution are 
those where 
the word~$w$ is weakly increasing and the subset
$Q$ consists of all non-left-to-right-maxima $r$
with $w_{r} \geq 1$.
Note that the weight of the pair $(w,Q)$ is the
weight
$q^{|Q|} \cdot q^{\sum_{r \in Q} w_{r}-1} = q^{\sum_{r \in Q} w_{r}}$.
Finally, 
the sum of the weights of these pairs is~$\qchoose{n}{k}{q}$.
\end{proof}

\section{An extension of Mercier's identities}
\label{section_extension_of_Mercier}

We simultaneously generalize two identities of Mercier by
introducing a two parameter identity
involving $q$-Stirling numbers.
The proof of this identity depends on
a different decomposition of  $RG$-words.

Let $([n]_{q})_{k}$ denote the $q$-analogue of the lower factorial,
that is,
$([n]_{q})_{k} = [n]_{q}!/[n-k]_{q}!$.
Alternatively, one can expand it as the product 
$$
([n]_{q})_{k} = [n]_{q} \cdot [n-1]_{q} \cdots [n-k+1]_{q} .
$$
\begin{theorem}
For three non-negative integers $n$, $r$ and $s$ such that
$s < r \leq n$ the following holds:
\begin{align}
\sum_{k=r}^{n}
(-q^{s})^{k-r}
\cdot ([k-s-1]_{q})_{k-r}
\cdot S_{q}[n,k]
& =
\sum_{i=r-1}^{n-1} S_{q}[i,r-1] \cdot [s]_{q}^{n-i-1}  . 
\label{equation_two_parameters}
\end{align}
\label{theorem_two_parameters}
\end{theorem}
\begin{proof}
On the set of $RG$-words
$S = \bigcupdotalt_{r \leq k \leq n} \rg{n}{k}$
define a weight function by
$$
f(w) = (-q^{s})^{k-r} \cdot ([k-s-1]_{q})_{k-r} \cdot \wt(w).
$$
Our objective is to evaluate
the sum
$\sum_{w \in S} f(w)$,
which is the left-hand side
of~\eqref{equation_two_parameters}.
We do this in two steps.
First we will partition the set $S$ into blocks
and extend the weight $f$ to a block $B$ by
$f(B) = \sum_{w \in B} f(w)$.
Secondly, on the set of blocks we will define a sign-reversing
involution such that
if two blocks $B$ and $C$ are matched,
their $f$-weights cancel, that is,
$f(B) + f(C) = 0$. Hence the right-hand side 
of~\eqref{equation_two_parameters}
will be equal to
the sum over all the blocks that have not been matched.

We now define an equivalence relation on the set $S$.
The blocks of our partition will be the equivalence classes.
For any integer $k$ where $r \leq k \leq n$,
we say two words $u, v \in \rg{n}{k}$ are equivalent
if there exists an index $i$ such that $s+1 \leq u_{i}, v_{i} \leq k$
and
$u = x \cdot u_{i} \cdot y$,
$v = x \cdot v_{i} \cdot y$
for some word $x \in \rg{i-1}{k}$
and $y$ a word in $[1,s]^{*}$.
Note that when
$s=0$ that $y$ is the empty word~$\epsilon$.
If a word $v \in \rg{n}{k}$ is of the form $v = x \cdot k \cdot y$
for some $x \in \rg{i-1}{k-1}$
and $y \in [1,s]^{*}$,
then this word is not equivalent to any
other words and hence it belongs to a singleton block.
Since for any $RG$-word we can find a decomposition described as
above, this is a partition of the set $\rg{n}{k}$ and hence the set $S$.

Match the singleton block
$B = \{x \cdot k \cdot y\}$
where $x \in \rg{i-1}{k-1}$, $y \in [1,s]^{*}$ and $r < k$
with the block
$$
C = \{x \cdot j \cdot y \: : \: s+1 \leq j \leq k-1\}.
$$
It is straightforward to check that $C \subseteq \rg{n}{k-1} \subseteq S$.
Note that the weight
$\wt(x \cdot j \cdot y)$ factors as
$\wt(x) \cdot q^{j-1} \cdot \ls(y)$.
Moreover, the $f$-weight of the block $C$ satisfies
\begin{align}
\label{equation_Stirling_f_weight}
\nonumber
f(C)
& =
\sum_{j=s+1}^{k-1}
(-q^{s})^{k-r-1} \cdot
([k-s-2]_{q})_{k-r-1}
\cdot \wt(x) \cdot q^{j-1} \cdot \ls(y)
\nonumber \\
& =
(-q^{s})^{k-r-1} \cdot 
([k-s-2]_{q})_{k-r-1}
\cdot \wt(x) \cdot q^{s} \cdot [k-s-1]_{q} \cdot \ls(y)
\nonumber \\
& =
- (-q^{s})^{k-r} \cdot
([k-s-1]_{q})_{k-r}
\cdot \wt(x) \cdot \ls(y)
\nonumber \\
& =
-f(x \cdot k \cdot y).
\end{align}
Hence the weight of the two blocks $B$ and $C$ cancel each other.

It remains to determine the weight
of the unmatched blocks.
Observe that every block of the form
$\{x \cdot j \cdot y : s+1 \leq j \leq k\}$
where $x \in \rg{i}{k}$ and $y \in [1,s]^{*}$ has been matched
by the above construction.
Hence the unmatched blocks are singleton blocks.
Given a word $u \in \rg{n}{k} \subseteq S$,
it has a unique factorization as $u = x \cdot j \cdot y$
where $y \in [1,s]^{*}$ and $s < j$.
If $x$ is a word in $\rg{i}{k}$ then the word $u$
belongs to a block that has been matched.
If $x \in \rg{i}{k-1}$ then
$m = k$ and the word $u$
belongs to a singleton block which has been matched if $k > r$.
Hence the unmatched blocks are of the form
$\{x \cdot r \cdot y\}$ where 
$x \in \rg{i}{r-1}$ and $y \in [1,s]^{n-i-1}$.
The sum of their $f$-weights are
\begin{align*}
\sum_{i=r-1}^{n-1}
\sum_{\substack{x \in \rg{i}{r-1} \\ y \in [1,s]^{n-i-1}}}
f(x \cdot r \cdot y)
& =
\sum_{i=r-1}^{n-1}
\sum_{\substack{x \in \rg{i}{r-1} \\ y \in [1,s]^{n-i-1}}}
\wt(x) \cdot \ls(y) \\
& =
\sum_{i=r-1}^{n-1}
S_{q}[i,r-1] \cdot [s]_{q}^{n-i-1} ,
\end{align*}
which is the right-hand side of the desired identity.
\end{proof}

Setting $(r,s) = (1,0)$
and $(r,s) = (2,1)$ in
Theorem~\ref{theorem_two_parameters}
we obtain two special cases,
both of which are due to
Mercier~\cite[Theorem~2]{Mercier_q_identities}.
The second identity is a $q$-analogue of
a result due to Jordan~\cite[equation~5, page~186]{Jordan}.

\begin{corollary}[Mercier, 1990]
\label{corollary_Mercier_involution}
For $n \geq 2$, the following two identities hold:
\begin{align}
\sum_{k=1}^{n} (-1)^k \cdot [k-1]_{q}! \cdot S_{q}[n,k] 
& = 0, \label{equation_Mercier_1}\\
\sum_{k=2}^{n} (-1)^k \cdot q^{k-2} \cdot [k-2]_{q}! \cdot S_{q}[n,k] 
& = n-1.   \label{equation_Mercier_2}
\end{align}
\end{corollary}

\noindent
Mercier's identity~\eqref{equation_Mercier_2} 
reappears in work of Ernst~\cite[Corollary~3.30]{Ernst}.

The next result is the case $r = s+1$ in 
Theorem~\ref{theorem_two_parameters}.
Here we provide a different expression.

\begin{proposition}
The following identity holds:
\begin{multline}
\sum_{k=r}^{n}
(-q^{r-1})^{k-r}
\cdot ([k-r]_{q})_{k-r}
\cdot S_{q}[n,k]
= \\
\sum_{c_{1} + c_{2} \plusdots c_{r-1} = n-r+1}
c_{r-1} \cdot [1]_{q}^{c_{1}} \cdot [2]_{q}^{c_{2}}
\cdots [r-2]_{q}^{c_{r-2}} \cdot [r-1]_{q}^{c_{r-1}-1} . 
\label{equation_long}
\end{multline}
\label{proposition_long}
\end{proposition}
\begin{proof}
The $f$-weights of the unmatched words $u = x \cdot r \cdot y$
in the proof of Theorem~\ref{theorem_two_parameters}
can be determined in a different manner.
Since $x \in \rg{i}{r-1}$ we can factor $x$
according to equation \eqref{equation_expansion}.
Hence the unmatched word $u$
has the form
$$ u = 1 \cdot x_{1} \cdot 2 \cdot x_{2} \cdot 3
               \cdots (r-1) \cdot x_{r-1} \cdot r \cdot y, $$
where $x_{i}$ belongs to $[1,i]^{*}$
and
$y$ belongs to $[1,r-1]^{*}$.
All possible words $x_{i}$ of length~$c_{i}$
give total weight of $[i]_{q}^{c_{i}}$ for $i \leq r-2$.
For the word $x_{r-1} \cdot r \cdot y$,
suppose its length is $c_{r-1}$.
Then we have $c_{r-1}$ choices to place the letter~$r$ 
and the total weight of such words of the form
$x_{r-1} \cdot y$ will be $[r-1]_{q}^{c_{r-1}-1}$.
Hence the $f$-weight for unmatched words
is given by
equation~\eqref{equation_long}.
\end{proof}

Recall the Stirling numbers of the second kind are specializations of
the homogeneous symmetric function;
see equation~\eqref{equation_complete_symmetric_functions}.
Thus one can view
Theorem~\ref{theorem_two_parameters} from a symmetric function
perspective.

\begin{theorem}
The following polynomial identity holds:
\begin{align*}
\sum_{i=0}^{n-r} h_{i}(x_{1}, x_{2},\ldots, x_{r-1})\cdot x_{s}^{n-r-i}
=
& 
\sum_{k=r}^{n}
(x_{s}-x_{r})\cdot (x_{s}-x_{r+1})\cdots (x_{s}-x_{k-1})
\cdot h_{n-k}(x_{1}, x_{2}, \ldots, x_{k}).
\end{align*}
\label{theorem_x}
\end{theorem}
\begin{proof}
Let $[t^{n}]f(t)$ denote the coefficient of $t^{n}$ in $f(t)$.
We consider the function
$$
G_{r}(t)
=
\frac{1}{1-x_{s}\cdot t} \cdot \prod_{j=1}^{r-1} \frac{1}{1-x_{j}\cdot t}
=
\prod_{j=1}^{r} \frac{1}{1-x_{j}\cdot t}
+
\frac{(x_{s}-x_{r})\cdot t}{1-x_{s}\cdot t}
\cdot
\prod_{j=1}^{r} \frac{1}{1-x_{j}\cdot t},
$$
and compute the coefficient of $t^{n-r}$ in two ways.
Using the first expression of $G_{r}(t)$
and that
$$
\frac{1}{1-x_{s}\cdot t}
=
\sum_{i\geq 0} x_{s}^{i} \cdot t^{i}
\;\; \text{ and } \;\;
\prod_{j=1}^{r-1} \frac{1}{1-x_{j}\cdot t}
= 
\sum_{i \geq 0} h_{i}(x_{1}, x_{2},\ldots, x_{r-1})\cdot t^{i},
$$
we obtain
\begin{equation}
[t^{n-r}]G_{r}(t)
=
\sum_{i=0}^{n-r} h_{j}(x_{1}, x_{2}, \ldots, x_{r-1})\cdot x_{s}^{n-r-i}.
\label{equation_symetric_one}
\end{equation}
Using the second expression of $G_{r}(t)$ we have
\begin{align}
\nonumber
[t^{n-r}]G_{r}(t)
&=
[t^{n-r}] \prod_{j=1}^{r} \frac{1}{1-x_{j} \cdot t}
+
[t^{n-r}] \frac{(x_{s}-x_{r})\cdot t}{1-x_{s}\cdot t}
\cdot
\prod_{j=1}^{r} \frac{1}{1-x_{j}\cdot t} \\
&=
h_{n-r}(x_{1}, x_{2}, \ldots, x_{r})
+
(x_{s}-x_{r}) \cdot [t^{n-r-1}] G_{r+1}(t) .
\label{equation_symmetric_iteration_one}
\end{align}
Iterate equation~\eqref{equation_symmetric_iteration_one}
$n-r$ times yields
\begin{equation}
\label{equation_symmetric_two}
[t^{n-r}]G_{r}(t)
=
\sum_{k=r}^{n}
(x_{s}-x_{r})\cdot (x_{s}-x_{r+1}) \cdots (x_{s}-x_{k-1})
\cdot h_{n-k}(x_{1}, x_{2}, \ldots, x_{k}).
\end{equation}
Now combine equations~\eqref{equation_symetric_one}
and~\eqref{equation_symmetric_two} we obtain the desired identity.
\end{proof}

\begin{proof}
[Second proof of Theorem~\ref{theorem_two_parameters}.]
Substituting $x_{i} = [i]_{q}$
in Theorem~\ref{theorem_x} yields the result
using that $[s]_q - [i]_q = -q^{s} \cdot [i-s]_q$ when $i > s \geq 0$.
\end{proof}

\section*{Acknowledgements}

The authors thank Dennis Stanton for directing us to the
pair of identities of Carlitz discussed in
Section~\ref{section_another_identity_from_Carlitz}.
The first author thanks the Simons Foundation
(grant \#206001)
for partially supporting two research visits
to Princeton University
to work with the second and third author.
The second author was partially supported by
NSA grant H98230-13-1-0280.  
This work was partially supported by grants
from the Simons Foundation
(\#429370 to Richard Ehrenborg; 
\#206001 and \#422467 to Margaret Readdy).
The authors would also like to thank the Princeton University
Mathematics Department for its hospitality
and support.

\newcommand{\arxiv}[3]{{\sc #1,} #2, {\tt #3}.}
\newcommand{\journal}[6]{{\sc #1,} #2, {\it #3} {\bf #4} (#5), #6.}
\newcommand{\book}[4]{{\sc #1,} ``#2,'' #3, #4.}
\newcommand{\bookf}[5]{{\sc #1,} ``#2,'' #3, #4, #5.}
\newcommand{\books}[6]{{\sc #1,} ``#2,'' #3, #4, #5, #6.}
\newcommand{\collection}[6]{{\sc #1,}  #2, #3, in {\it #4}, #5, #6.}
\newcommand{\thesis}[4]{{\sc #1,} ``#2,'' Doctoral dissertation, #3, #4.}
\newcommand{\springer}[4]{{\sc #1,} ``#2,'' Lecture Notes in Math.,
                          Vol.\ #3, Springer-Verlag, Berlin, #4.}
\newcommand{\preprint}[3]{{\sc #1,} #2, preprint #3.}
\newcommand{\preparation}[2]{{\sc #1,} #2, in preparation.}
\newcommand{\appear}[3]{{\sc #1,} #2, to appear in {\it #3}}
\newcommand{\submitted}[3]{{\sc #1,} #2, submitted to {\it #3}}
\newcommand{\JCTA}{J.\ Combin.\ Theory Ser.\ A}
\newcommand{\AdvancesinMathematics}{Adv.\ Math.}
\newcommand{\JournalofAlgebraicCombinatorics}{J.\ Algebraic Combin.}

\newcommand{\communication}[1]{{\sc #1,} personal communication.}


\bigskip

\noindent
{\em Y.\ Cai,
Department of Mathematics,
Texas A\&M University,
College Station, TX 77843-3368,}
{\tt ycai@math.tamu.edu}.

\noindent
{\em R.\ Ehrenborg and M.\ Readdy,
Department of Mathematics,
University of Kentucky,
Lexington, KY 40506-0027,}
{\tt {richard.ehrenborg, margaret.readdy}@uky.edu}.

\end{document}